\documentclass[10pt, leqno]{amsart}

\usepackage[utf8]{inputenc}
\usepackage{amssymb}
\usepackage{amsmath,amscd}
\usepackage[initials,nobysame,alphabetic]{amsrefs}
\usepackage{amsmath, amssymb}
\usepackage{amsfonts}
\usepackage{mathrsfs}
\usepackage{mathpazo}
\usepackage{mathtools}
\usepackage[arrow,matrix,curve,cmtip,ps]{xy}
\usepackage{amsthm}
\usepackage{float}
\usepackage{hyperref}
\usepackage{graphics}
\usepackage{graphicx}
\usepackage{verbatim}
\usepackage{multirow}
\usepackage{tikz}
\usepackage{enumerate}
\allowdisplaybreaks

\newtheorem{theorem}{Theorem}[section]

\newtheorem{proposition}[theorem]{Proposition}
\newtheorem{corollary}[theorem]{Corollary}

\newtheorem*{theorem*}{Theorem}
\newtheorem{remark}[theorem]{Remark}

\numberwithin{equation}{section}

\def\th{\theta}                    

\def\z{\zeta}                  

\usepackage{color}

\newtheorem{Remark}{Remark}[part]

\def \t{\tau}

\renewcommand{\baselinestretch}{1.1}
\newcommand{\nc}{\newcommand}
\nc{\esssup}{\mathop{\mathrm{ess\,sup}}}
\nc{\essinf}{\mathop{\mathrm{ess\,inf}}}
\nc{\argmax}{\mathop{\mathrm{arg\,max}}}
\def \lb {\label}

\def \nd {\noindent}
\def \ms {\medskip}

\def \n{\nu}
\def \q {\quad}
\def \qq{\qquad}
\def \P{\mathbb{P}}

\def \R{\mathbb{R}}

\def \E{\mathbb{E}}
\def \F{\mathbb{F}}
\def \U{\mathbb{U}}

\def \D{\mathbb{D}}

\def \1{\mathbf{1}}

\def \Fc{\mathcal{F}}

\def \eps{\varepsilon}
\def \nn {\nonumber}
\def \ed {\end{document}}

\def \b {\beta}

\def\enqs{\end{eqnarray*}}
\def\beq{\begin{eqnarray}}
\def\enq{\end{eqnarray}}

\def \d {\delta}

\def \om {\Omega}
\def \dl {\Delta}
\def \mx {\mbox}
\def \ca {\mathcal {A}}
\def \co {{\mathcal O}}
\def \be {\begin{equation}}
\def \ee {\end{equation}}
\def \rw {\rightarrow}
\def \rwi {{\rw +\infty}}
\def \Dl {\Delta}
\def \ct{ c_{\theta} }

\def \nx{(\nu,\xi)}
\def \rwf{\rightarrow \infty}
\def \bn {\bar{\mathbb{N}}}
\def \nb {\mathbb{N}}
\def \bbn{\overline{\nb}}

\begin{document}

\title{Discrete time Stochastic Impulse Control with Delay}
\author{Boualem Djehiche}
\thanks{B. Djehiche is partially supported by the Lebesgue Center of Mathematics "Investissements d’avenir" program-ANR-11-LABX-0020-01}
\address{Department of Mathematics \\ KTH Royal Institute of Technology \\ 100 44, Stockholm \\ Sweden}
\email{boualem@kth.se}
\author{Said Hamad{\`e}ne}
\address{LMM, Le Mans University, Avenue Olivier Messiaen, 72085 Le Mans, Cedex 9, France} 
\email{hamadene@univ-lemans.fr}
\date{\today}


\maketitle

\begin{abstract}
We study a class of infinite-horizon impulse control problems with execution delay in discrete time. Using probabilistic methods, particularly the notion of the Snell envelope of processes, we construct an optimal strategy among all admissible strategies for both risk-neutral and risk-sensitive utility functions. Furthermore, we establish the existence of bounded $\varepsilon$-optimal strategies. This framework provides a robust approach to handling execution delays in discrete-time stochastic systems.
\end{abstract}

\tableofcontents
\section{Introduction}

Impulse control is an essential part of stochastic control theory, where control affects both the timing and the jump size of the state because interventions occur at discrete times. These problems naturally arise in many practical situations where actions cannot be taken continuously but must be implemented at specific moments. Examples include financial decision-making, inventory control, and energy management systems, where optimal strategies must account for both timing and magnitude of interventions.

The formal study of impulse control began with the pioneering work of Bensoussan and Lions \cite{Bensoussan84}, who introduced in a continuous time setting a rigorous framework based on quasi-variational inequalities (QVI). They established a connection between impulse control, variational inequalities, and the dynamic programming principle, laying the foundation for further developments. In the continuous-time setting, QVIs provide a powerful mathematical tool to characterize optimal strategies in terms of stopping times and intervention decisions. 

While continuous-time models are analytically appealing, a discrete-time formulation is often more suitable for real-world applications, particularly in scenarios where decision-making occurs at a finite number of times. Bensoussan \cite{bensoussan08} considers impulse control in a discrete time Markovian set-up, using a dynamic programming framework, which provides the possibility of obtaining analytic solutions for applications to inventory control and portfolio choice under constraints.

An additional layer of complexity arises in impulse control when execution delays are introduced. Execution delay refers to the time lag between the moment a control action is decided and its actual implementation. This phenomenon is ubiquitous in practice, for example, in financial markets where trades take time to clear or in manufacturing systems where orders experience processing delays. Delays can significantly affect optimal decision making, as strategies must account for not only current system states but also the anticipated evolution during the delay period. 

 Bruder and Pham \cite{Bruder} tackled this challenge for finite-horizon impulse control problems with execution delays in a Markovian setting. They use the dynamic programming principle to derive the Hamilton-Jacobi-Bellman equation satisfied by the value function of the problem.

In the context of processes with memory effects, Djehiche, Hamadène, and Hdhiri \cite{Djehiche10} extended impulse control to non-Markovian processes. They established the existence of optimal strategies, using the Snell envelope in conjunction with reflected BSDEs. Their results emphasize the importance of incorporating path dependencies and history when modeling real-world systems subject to delays.

Risk-sensitive control problems, which address decision-making under e.g. risk aversion, further complicate the impulse control framework. Hdhiri and Karouf \cite{Hdhiri2} considered finite- horizon  risk-sensitive impulse control of non-Markovian processes, incorporating execution delays and random coefficients. Their work shows that optimal solutions must carefully balance the risk-sensitive nature of the utility function with the inherent uncertainties of delayed control. Infinite-horizon stochastic impulse control with delay and random
coefficients is studied in \cite{Djehiche22}.

Optimal stochastic control problems involving delays are not restricted to standard Markovian or non-Markovian processes. For example, Robin \cite{Robin1} studied impulse control for Markov processes with delays as early as 1976, providing foundational insights into the role of delay in controlling stochastic systems. Similarly, Cadenillas and Zapatero \cite{cadzap2}, Palczewski and Stettner \cite{Palczewski}, and Oksendal and Sulem \cite{Oksendal2}, among many other authors, applied impulse control techniques to portfolio optimization, showing their relevance to financial mathematics and the importance of delays in practical decision making.

Despite significant progress, much of the existing literature has focused on finite-horizon problems in a continuous-time setting. However, many applications require considering infinite-horizon problems, where interventions must be optimized over an indefinite future. This introduces additional technical challenges, including ensuring the existence of optimal strategies and bounding long-term performance in the presence of delays.

In this paper, we address infinite-horizon impulse control with execution delays in discrete time. Specifically, we study a class of impulse control problems in which decisions are implemented after a finite delay $\dl$, and the system evolves in discrete time. We do not require a specific sign for the running payoff function $g$ and the impulse payoff function $\Psi$. When $\Psi(\xi)$ is negative, it means that the impulse $\xi$ provides a subsidy rather than being a cost to the decision maker. On the other hand, between two successive impulses, the dynamics of the controlled system is no longer Markovian, contrary to the setting in \cite{bensoussan08}. Our analysis uses probabilistic tools, particularly the Snell envelope of processes, to establish the existence of optimal strategies under both risk-neutral and risk-sensitive utility functions. Furthermore, we show the existence of bounded $\varepsilon$-optimal strategies, providing a robust framework for practical applications.

The remainder of this paper is organized as follows. Section 2 introduces the problem formulation, including the dynamics of the system and admissible strategies. Section 3 presents our main results, including the existence of optimal and bounded $\varepsilon$-optimal strategies. Section 4 discusses the risk-sensitive case and highlights key differences from the risk-neutral case. Section 5 provides concluding remarks.

\section{The model}
\def \bbn {\overline{\mathbb{N}}}
\def \bn {\mathbb{N}}
Let $(\om,\Fc,\P)$ be a probability space and $\F=(\Fc_n)_{n\ge 1}$ a filtration of $(\om,\Fc)$. Let $\dl$ be a positive integer and $\U$ be a finite subset of $\R^d$. 
Denote by $\Fc_\infty:=\sigma\{\Fc_n,n\ge 0\}$ the $\sigma$-algebra generated by $\Fc_n$, $n\ge 0$, and for any $\F$-stopping time $\t$, set $\Fc_\t:=\{A\in \F_\infty, A\cap \{\t=n\}\in \Fc_n, \mbox{ for any }n\ge 0\}$. Let $\overline{\mathbb{N}}:=\mathbb{N}\cup \{+\infty\}$ and $(\theta_n)_{n\ge 0}$ be a stochastic process indexed by $\bn$. We set $\theta^-_\infty:=\limsup_{n\rwi}\theta_n$ and extend $(\theta_n)_{n\ge 0}$ to $\bbn$ by setting $\theta_\infty=\theta^-_\infty$.

Let $(X_n)_{n\ge 0}$ be an $\F$-adapted stochastic process valued in $\R^d$ that describes the evolution of a system when it is not controlled. An impulse strategy with delay $\dl$ is a sequence $\d:=(\t_p,\z_p)_{p\ge 0}$, where for any $p\ge 0$, 
\begin{itemize}
\item[(a)]  $\t_p$ is an $\F$-stopping time such that $\t_{p+1}-\t_p\ge \dl$, $\,\,\P$-a.s.;

\item[(b)] $\z_p$ is a $\U$-valued and $\F_{\t_p}$-measurable random variable.
\end{itemize}
\nd In this definition, $\t_p$ is the time when the controller of the system decides to make an impulse of magnitude $\z_p$. The sequence $\delta = (\tau_p, \xi_p)_{p \geq 0}$ is called an admissible strategy of the impulse control problem and the set of admissible strategies is denoted by $\mathcal{A}$.

When the controller implements an impulse strategy  $\d:=(\t_p,\z_p)_{p\ge 0}$, the evolution of the controlled system is given by the process $X^\d:=(X^\d_k)_{k\ge 0}$ defined by
\begin{equation*}
X^\d_k=\begin{cases} X_k & \mx{ if }k<\t_0+\dl, \\
X_k+\z_0+\cdots+\z_\ell & \mbox{ if }\t_\ell+\dl \le k<\t_{\ell+1}+\dl,
\end{cases}
\end{equation*} or in a compact form,
\begin{equation}\label{X-delta}
X^\d_k=X_k 1_{\{k<\t_0+\dl\}}+\sum_{\ell \ge 0}(X_k+\z_0+\dots+ \z_\ell) 1_{\{\t_\ell+\dl \le k<\t_{\ell+1}+\dl\}},\quad k\ge 0.
\end{equation}
\section{The risk-neutral impulse control problem}\lb{rn-case}
\def \bn {\bar{\mathbb{N}}}
\def \nb {\mathbb{N}}
The payoff which is a reward for the controller is given by the following risk-neutral utility function:
\be\lb{rn}
J(\d):=\E[\sum_{k\ge 0}e^{-\th k}g(X^\d_k)-\sum_{\ell \ge 0}e^{-\th (\ell+\dl)}\Psi(\z_{\ell})],
\ee
where 
\begin{itemize}
    \item[(i)] $\th$, standing for the discount factor, is a positive constant, 
    \item[(ii)] $g$ and $\Psi$ are bounded functions. They stand respectively for the running payoff and the accumulated payoff due to impulses. 
\end{itemize}
The controller's problem is to find an optimal strategy $\d^*$, that is, which satisfies
$$
J(\d^*)=\sup_{\d \in \ca}J(\d).
$$
\begin{Remark}
We do not require a specific sign for $g(\cdot)$ and more importantly for $\Psi(\cdot)$ because we allow for a negative impulse payoff (cost), which means that an impulse induces a reward instead of incurring a cost. In some way, in the latter situation, the controller gets subsidies.  \qed
\end{Remark}

\subsection{Iterative scheme}
Let $\nu$ be  an $\F$- stopping time and $\xi$ be a $\U$-valued and $\Fc_\n$-measurable random variable. For any  $n\ge 0$ and any pair 
$(\nu,\xi)$, let $(Y^n_k(\nu,\xi))_{k\in \nb}$ be the sequence of processes defined recursively as follows. For any $k\in \nb$,
\begin{equation}\lb{Y-0-k}
Y^0_k(\nu,\xi)=\E[\sum_{\ell \ge k}e^{-\th \ell}g(X_\ell+\xi)1_{\{\ell \ge \n\}}|\Fc_k].
\end{equation}
Since $g(\cdot)$ is bounded,
$(Y^0_k(\nu,\xi))_{k\in \nb}$ is well-posed. Moreover, 
$Y^{0,-}_\infty(\nu,\xi)=\underset{k\rwf}{\lim}\,Y^0_k(\nu,\xi)=0$, $\P$-a.s., and in $L^p(\P)$, for any $p\in [1,\infty)$, since for any $k\in \nb$, 
\be \lb{estyo}|Y^0_k(\nu,\xi)|\le (1-e^{-\theta})^{-1}e^{-\theta k}\|g\|_{\infty},
\ee
where $\|g\|_{\infty}=\underset{x\in \R^d}{\sup}|g(x)|$. Then, we set $Y^{0}_{\infty}(\nu,\xi)=0$, which makes the sequence $(Y^n_k(\nu,\xi))_{k\in \bbn}$ well-posed. 
Next, for $n\ge 1$ and any pair $\nx$, we set 
\begin{equation}\label{obst1} \begin{array}{lll}
\co^n_k(\nu,\xi):=\left\{\begin{array}{l}
\E[\underset{k+1\le \ell < k+\dl}{\sum}e^{-\th \ell}g(X_\ell+\xi)1_{\{\ell \ge \n\}}|\Fc_k] \\ \qquad\qquad\quad+\underset{\beta \in \U}{\max}\,\E[-e^{-\th (k+\dl)}\Psi(\beta)+Y^{n-1}_{k+\dl}(\n,\xi+\beta)|\Fc_k]\,\, \mbox{ if }k\in \nb,\\
0 \,\,\mbox{ if }k=+\infty,
\end{array}\right.
\end{array}
\end{equation}
and for any $k\in \bbn$, 
\be \lb{eq6-n-k}
Y_k^{n}(\nu,\xi):=\esssup_{\t\ge k}
\E[\sum_{k\le 
\ell\le\t}e^{-\th \ell}g(X_\ell+\xi)1_{\{\ell \ge \n\}}+\co^{n}_\t(\nu,\xi)|\Fc_k],
\ee
where the essential supremum is taken over $\F$-stopping times $\t\ge k$. Note that according to this definition $Y_\infty^{n}(\nu,\xi)=\co^n_\infty(\nu,\xi)=0.$
\medskip

In the next proposition we collect some properties of the sequence $(Y_k^{n}(\nu,\xi))_{k\in \bbn}$, $n\ge 0$.
\def \bbn{\overline{\nb}}
\begin{proposition}\label{P0} For any $n\ge 0$ and any pair $(\n,\xi)$, the processes $(Y_k^{n}(\nu,\xi))_{k\in \bbn}$ are well-posed and 
$Y_\infty^{n,-}(\nu,\xi)=\lim_{k \rwf}Y_k^{n}(\nu,\xi)=Y_\infty^{n}(\nu,\xi)=0.$
\end{proposition}
\begin{proof} We proceed by an induction argument on $n$, to show that for any $n\ge 0$ and any pair $\nx$, the process $(Y_k^{n}(\nu,\xi))_{k\in \bbn}$ exists and satisfies, for any $k\in \nb$ and any pair $\nx$, 
\be \lb{proprec}
|Y^n_k\nx|\le ((2n+1)\ct \|g\|_\infty+n|\Psi|_{\infty})e^{-\theta k}.
\ee where $\ct:=(1-e^{-\theta})^{-1}$.

For $n=0$, the property holds true since 
$(Y_k^{0}(\nu,\xi))_{k\in \bbn}$ is well-posed and by \eqref{estyo}, the estimate \eqref{proprec} is satisfied. 
Assume now that for some $n\ge 0$ and for any $\nx$, the process $(Y_k^{n}(\nu,\xi))_{k\in \bbn}$ is well-posed and \eqref{proprec} is satisfied. By definition, the process $(\co_k^{n+1}(\nu,\xi))_{k\in \nb}$ is given by
\begin{equation}\label{obst} \begin{array}{lll}
 \co^{n+1}_k(\nu,\xi):=
\E[\underset{k+1\le \ell < k+\dl}{\sum}e^{-\th \ell}g(X_\ell+\xi)1_{\{\ell \ge \n\}}|\Fc_k] \\ \qquad\qquad\qquad\qquad\qquad\qquad+\underset{\beta \in \U}{\max}\,\E[-e^{-\th (k+\dl)}\Psi(\beta)+Y^{n}_{k+\dl}(\n,\xi+\beta)|\Fc_k], \quad k\in \nb.
\end{array}
\end{equation}
The boundedness of $g(\cdot)$, $\Psi(\cdot)$ and the inequality \eqref{proprec} imply that, for any $k\in \nb$, $\co^{n+1}_k(\nu,\xi)$ is well-posed. Thus, the process 
$(\co^{n+1}_k(\nu,\xi))_{k\in \bbn}$ is also well-posed since $\co^{n+1}_\infty(\nu,\xi)=0$. Moreover, for any $k\in \nb$, 
$$
|\co^{n+1}_k(\nu,\xi)|\le ((2n+2)\ct \|g\|_\infty+(n+1)|\Psi|_{\infty})e^{-\theta k}.$$
So, the process $(Y_k^{n+1}(\nu,\xi))_{k\in \bbn}$ is well-posed and satisfies
$$
|Y^{n+1}_k\nx|\le ((2n+3)\ct \|g\|_\infty+(n+1)|\Psi|_{\infty})e^{-\theta k}, k\in \nb. 
$$
Consequently, the property is valid for $n+1$ and hence it is valid for every $n\ge 0$. 

Finally, the fact that $Y_\infty^{n,-}(\nu,\xi)=\lim_{k \rwf}Y_k^{n}(\nu,\xi)=Y_\infty^{n}(\nu,\xi)=0$ follows immediately from \eqref{proprec}.
\end{proof}
\begin{remark}\lb{domon} As a by-product, we obtain that, for any $n\ge 0$ and any $\nx$
\be \lb{estokn}
|\co^{n}_k(\nu,\xi)|\le \left(2n\ct \|g\|_\infty +n\|\Psi\|_{\infty}\right)e^{-\theta k},\quad k\ge 0,
\ee 
where 
$\|\Psi\|_{\infty}=\sup_{u\in \U}|\Psi(u)|$. Therefore, $\co^{n,-}_\infty(\nu,\xi)=0$, for any $n\ge 1$ and any pair $\nx$. \qed
\end{remark}

For every $n\ge 1$ and every $(\nu,\xi)$, set
\be\lb{Y-tilde-n}
\widetilde{Y}_k^n(\nu,\xi):=Y^n_k(\nu,\xi)+\sum_{\ell \le k-1}e^{-\th \ell}g(X_\ell+\xi)1_{\{\ell \ge \n\}},\,\quad k\ge 0.
\ee
In view of \eqref{eq6-n-k}, we have
\be \lb{eq6-tilde-n}
\widetilde{Y}_k^{n}(\nu,\xi)=\esssup_{\t\ge k}
\E[\sum_{\ell\le\t}e^{-\th \ell}g(X_\ell+\xi)1_{\{\ell \ge \n\}}+\co^{n}_\t(\nu,\xi)|\Fc_k],\quad k\ge 0,
\ee
which means that it is the Snell envelope (see e.g. \cite{Dellacherie}, p.431 or \cite{El-Karoui1}, p.140) of the obstacle processes 
\begin{equation}\lb{L-k-n}
L^n_k\nx:=\sum_{\ell\le k}e^{-\th \ell}g(X_\ell+\xi)1_{\{\ell \ge \n\}}+\co^{n}_k(\nu,\xi),\quad k\ge 0, \quad L^n_\infty\nx=\underset{k\to\infty}{\lim}\, L^n_k\nx,
\end{equation} 
i.e., $(\widetilde{Y}_k^n(\nu,\xi))_{k\ge 0}$  is the smallest supermartingale which dominates $(L_k\nx)_{k\ge 0}$. Thus, $(\widetilde{Y}_k^n(\nu,\xi))_{k\ge 0}$ is the unique process that satisfies the following recursive formula:
\be \lb{eq2-tilde-n}\left\{\begin{array}{l}
\widetilde{Y}^n_k(\nu,\xi)=\max\{\E[\widetilde{Y}^n_{k+1}(\nu,\xi)|\Fc_k], \underbrace{\sum_{\ell\le k}e^{-\th \ell}g(X_\ell+\xi)1_{\{\ell \ge \n\}}+\co^n_k(\nu,\xi)}_{L^n_k\nx}\},\quad  k\ge 0;
\\
\widetilde{Y}^n_\infty(\nu,\xi)=L^n_\infty\nx=
\sum_{\ell\ge 0}e^{-\th \ell}g(X_\ell+\xi)1_{\{\ell \ge \n\}}.
\end{array}\right.
\ee
Note that the r.v. $L^n_\infty\nx$ exists since the function $g(\cdot)$ is bounded and $\lim_{k\rw \infty}\co^{n}_k(\nu,\xi)=0.$
\begin{proposition}[Monotonicity property]\label{P1}
For any $n\ge0$ and any pair $(\nu,\xi)$, 
\begin{equation} \label{monotonicity}
Y^n_k(\nu,\xi)\le Y^{n+1}_k(\nu,\xi),\quad k\ge 0.
\end{equation}
\end{proposition}

\nd\begin{proof} We proceed by induction on $n$. Indeed, by \eqref{eq6-n-k}, for any $m\ge k$, we have
$$
Y^1_k(\nu,\xi)\ge \E[\sum_{k\le \ell \le m}e^{-\th \ell}g(X_\ell+\xi)1_{\{\ell \ge \n\}}+\co^{1}_m(\nu,\xi)|\Fc_k].
$$
Therefore, in view of Remark \ref{domon}, the boundedness of $g(\cdot)$ and conditional dominated convergence, it holds that, for any pair $(\nu,\xi)$, 
\begin{align*}
Y^1_k(\nu,\xi)&\ge  \underset{m\to \infty}{\lim}\,\E[\underset{k\le \ell \le m}{\sum}\,e^{-\th \ell}g(X_\ell+\xi)1_{\{\ell \ge \n\}}+\co^{1}_m(\nu,\xi)|\Fc_k] \\ & =\E[\underset{\ell \ge k}{\sum}\,e^{-\th \ell}g(X_\ell+\xi)1_{\{\ell \ge \n\}}|\Fc_k]=Y^0_k(\nu,\xi).
\end{align*}
Thus, the property holds for $n=0$. Now, suppose that for some $n$ and for any $\nx$ and any $k\ge 0$, $Y^n_k(\nu,\xi)\le Y^{n+1}_k(\nu,\xi)$. Then, by \eqref{obst},
it follows that $\co^{n+1}_k(\nu,\xi)\le \co^{n+2}_k(\nu,\xi)$ which implies that $Y^{n+1}_k(\nu,\xi)\le Y^{n+2}_k(\nu,\xi)$. Thus, the property holds for any $n\ge 0$, which means that $(Y^n\nx)_{n\ge 0}$ is non-decreasing in $n$.
\end{proof}

Next, we provide a characterization of the processes $Y^{n}(\nu,\xi)$ and state some of their properties. 

\begin{proposition}[Consistency property]\label{cs}
Let $\nu$ be a stopping time and $\xi$ be a $\U$-valued and $\Fc_{\nu}$-measurable r.v. Then, for any $\F$-stopping time $\nu'\ge \nu$, $\P$-a.s, it holds that, for any $n\ge 0$, 
$$Y^{n}_k(\nu,\xi)=Y^{n}_k(\nu',\xi), \quad \,\,k\ge \nu'.$$
\end{proposition}
\nd \begin{proof} We proceed by induction on $n$. For any $\nx $, 
$$
Y^0_{k\vee \nu'}(\nu,\xi)=\E[\sum_{\ell \ge {k\vee \nu'}}e^{-\th \ell}g(X_\ell+\xi)1_{\{\ell \ge \n\}}|\Fc_{{k\vee \nu'}}],\quad k\ge 0.
$$
But, since $\nu'\ge \nu$ $\P$-a.s., we must have 
$$
\E[\sum_{\ell \ge {k\vee \nu'}}e^{-\th \ell}g(X_\ell+\xi)1_{\{\ell \ge \n\}}|\Fc_{{k\vee \nu'}}]=\E[\sum_{\ell \ge {k\vee \nu'}}e^{-\th \ell}g(X_\ell+\xi)1_{\{\ell \ge \n'\}}|\Fc_{{k\vee \nu'}}]=Y^0_{k\vee \nu'}(\nu',\xi).
$$
Thus, the property is valid for $n=0$. Assume now that it is valid for some $n\ge 0$, i.e., for any pair $(\nu,\xi)$ and any $k\ge 0$, $Y^{n}_{k\vee \nu'}(\nu,\xi)=Y^{n}_{k\vee \nu'}(\nu',\xi)$. We have
\begin{equation}\begin{array}{ll}
\co^{n+1}_{k\vee \nu'}(\nu,\xi)&=
\E[\underset{{k\vee \nu'}+1\le \ell < {k\vee \nu'}+\dl}{\sum}\,e^{-\th \ell}g(X_\ell+\xi)1_{\{\ell \ge \n\}}|\Fc_{k\vee \nu'}]\\& \qquad \qquad +\underset{\beta \in \U}{\max}\,\E[-e^{-\th ({k\vee \nu'}+\dl)}\Psi(\beta)+Y^{n}_{{k\vee \nu'}+\dl}(\n,\xi+\beta)|\Fc_{k\vee \nu'}] \\ \\&=\E[\underset{{k\vee \nu'}+1\le \ell < {k\vee \nu'}+\dl}{\sum}\,e^{-\th \ell}g(X_\ell+\xi)1_{\{\ell \ge \n'\}}|\Fc_{k\vee \nu'}]\\\\& \qquad \qquad +\underset{\beta \in \U}{\max}\,\E[-e^{-\th ({k\vee \nu'}+\dl)}\Psi(\beta)+Y^{n}_{{k\vee \nu'}+\dl}(\n',\xi+\beta)|\Fc_{k\vee \nu'}]\\
&=\co^{n+1}_{k\vee \nu'}(\nu',\xi).
\end{array}
\end{equation}
The second equality is valid by the induction assumption, the fact that $\beta$ is constant and so $\xi+\beta$ is $\Fc_{\nu}$-measurable, and $\nu'\ge \nu$. But, by \eqref{eq6-n-k}, 
we have, for any $k\ge 0$, 
\be \lb {eq7xx}\begin{array}{ll}
Y_{k\vee \nu'}^{n+1}(\nu,\xi)&=\underset{\t\ge {k\vee \nu'}}{\esssup}\,
\E[\sum_{{k\vee \nu'}\le \ell\le\t}e^{-\th \ell}g(X_\ell+\xi)1_{\{\ell \ge \n\}}+\co^{n+1}_\t(\nu,\xi)|\Fc_{k\vee \nu'}]\\\\
&=\underset{\t\ge {k\vee \nu'}}{\esssup}\,
\E[\sum_{{k\vee \nu'}\le \ell\le\t}e^{-\th \ell}g(X_\ell+\xi)1_{\{\ell \ge \n'\}}+\co^{n+1}_\t(\nu',\xi)|\Fc_{k\vee \nu'}]
\\\\
&=Y_{k\vee \nu'}^{n+1}(\nu',\xi),\end{array}
\ee
which implies that the property holds for $n+1$, hence it is valid for every $n\ge 0$.
\end{proof}

For a given  $n\ge 0$, let $\ca_n$ be the set of admissible strategies $\d=(\t_p,\z_p)_{p\ge 0}$ such that $\t_n=+\infty$, $\P$-a.s. and, for $k\ge 0$, let $\ca_n^k$ be the subset of $\ca_n$ of strategies $\d=(\t_p,\z_p)_{p\ge 0}$ such that $\t_0\ge k$. When a strategy $\d$ belonging to $\ca_n$ is implemented, it means that only $n$ impulses, at most, are allowed. 

\ms We have
\begin{proposition}\lb{P2}
For any $n\ge 0$ and any pair $(\nu,\xi)$, 
$$
Y^n_k(\nu,\xi)=\underset{\d \in \ca^k_n}{\esssup}\,\E[\sum_{\ell\ge k}e^{-\th \ell}g(X^\d_\ell+\xi)1_{\{\ell \ge \nu\}}-\sum_{p \ge 0}e^{-\th (\t_p+\dl)}\Psi(\z_{p})|\Fc_k],\quad  k\ge \nu.
$$
\end{proposition}
\nd \begin{proof} First note that for $n=0$, the property is obvious since for any strategy $\d=(\t_p,\z_p)_{p\ge 0}$ such that $\t_0=+\infty$ $\,\P$-a.s., meaning that the controller does not exercise any impulse i.e., $X^\d=X$ for any $\d \in \ca_0$. Therefore, for any $k\ge 0$,
$$
Y^0_k(\nu,\xi)=
\E[\sum_{\ell \ge k}e^{-\th \ell}g(X_\ell+\xi)1_{\{\ell \ge \n\}}|\Fc_k]=\esssup_{\d \in \ca_0^k}\E[\sum_{\ell\ge k}e^{-\th \ell}g(X^\d_\ell+\xi)1_{\{\ell \ge \nu\}}-\sum_{\ell \ge 0}e^{-\th (\t_\ell+\dl)}\Psi(\z_{\ell})|\Fc_k].
$$
Fix $n\ge 1$. For $p=0,\ldots,n-1$, let $(\t^*_p,\b^*_p)$ be the pairs defined as follows (we omit the dependence on $n$ as there is no risk for confusion):
$$
\t^*_0=\inf\{\ell \ge k\vee \nu,\,\, 
Y^n_\ell(\nu,\xi)=\co^n_\ell(\nu,\xi)\}
$$
and $\b^*_0$ be an $\U$-valued and $\Fc_{\t^*_0}$-r.v.  such that
\begin{align*}
\co^{n}_{\t_0^*}(\nu,\xi)& =
\E[\underset{\t^*_0+1\le \ell < \t^*_0+\dl}{\sum}\,e^{-\th \ell}g(X_\ell+\xi)1_{\{\ell \ge \n\}}|\Fc_{\t^*_0}] \\ & \qquad \qquad  +\max_{\beta \in \U}\E[-e^{-\th (\t^*_0+\dl)}\Psi(\beta)+Y^{n-1}_{\t^*_0+\dl}(\n,\xi+\beta)|\Fc_{\t^*_0}]\\
& =\E[\underset{\t^*_0+1\le \ell < \t^*_0+\dl}{\sum}\,e^{-\th \ell}g(X_\ell+\xi)1_{\{\ell \ge \n\}}|\Fc_{\t^*_0}] \\ & \qquad \qquad \qquad  +\max_{\beta \in \U}\E[-e^{-\th (\t^*_0+\dl)}\Psi(\beta)+Y^{n-1}_{\t^*_0+\dl}(\t^*_0,\xi+\beta)|\Fc_{\t^*_0}]\\
& =\E[\underset{\t^*_0+1\le \ell < \t^*_0+\dl}{\sum}\,e^{-\th \ell}g(X_\ell+\xi)|\Fc_{\t^*_0}]+\E[-e^{-\th (\t^*_0+\dl)}\Psi(\beta^*_0)+Y^{n-1}_{\t^*_0+\dl}(\t^*_0,\xi+\beta^*_0)|\Fc_{\t^*_0}].
\end{align*}
Finally, we set $\t_n^*=+\infty$. 
First, note that $\t^*_0$ is optimal for \eqref{eq6-n-k} 
since, by \eqref{estokn}, 
$\underset{\ell \rwi}{\lim}\co^n_\ell(\nu,\xi)=\co^n_\infty(\nu,\xi)=0$ (see e.g. \cite{hamahassa}, p.184 for more details). On the other hand, the second equality is valid by the consistency condition of
\def \Dl {\Delta}
Proposition \ref{cs}, since $Y^{n-1}_{\t^*_0+\dl}(\t^*_0,\xi+\beta)=Y^{n-1}_{\t^*_0+\dl}(\n,\xi+\beta)$ as $\t_0^*\ge \nu$ and $\xi+\beta$ is $\Fc_\nu$-measurable because $\xi$ is $\Fc_\nu$-measurable and $\beta$ is a constant. Therefore, for any $k\ge \nu$, 
\begin{align}
\lb {eq7}
Y_k^{n}(\nu,\xi)\nonumber& =
\E[\sum_{k\le \ell\le\t_0^*}e^{-\th \ell}g(X_\ell+\xi)1_{\{\ell \ge \n\}}+\co^{n}_{\t_0^*}(\nu,\xi)|\Fc_k]\\
\nonumber&=\E[\sum_{k\le \ell\le\t_0^*}e^{-\th \ell}g(X_\ell+\xi)1_{\{\ell \ge \n\}}+\sum_{\t^*_0+1\le \ell < \t^*_0+\dl}e^{-\th \ell}g(X_\ell+\xi)1_{\{\ell \ge \n\}}\\\nonumber&\qquad \qquad -e^{-\th (\t^*_0+\dl)}\Psi(\beta^*_0)+Y^{n-1}_{\t^*_0+\dl}(\t^*_0,\xi+\beta^*_0)|\Fc_k]\\&=\E[\sum_{k\le \ell<\t_0^*+\dl}e^{-\th \ell}g(X_\ell+\xi)-e^{-\th (\t^*_0+\dl)}\Psi(\beta^*_0)+Y^{n-1}_{\t^*_0+\dl}(\t^*_0,\xi+\beta^*_0)|\Fc_k].
\end{align}
We shall repeat the same argument with $Y^{n-1}_{\t^*_0+\dl}(\t^*_0,\xi+\beta^*_0)$. Let 
$$
\t^*_1=\inf\{\ell \ge \t^*_0+\dl, \,\,
Y^{n-1}_\ell(\t^*_0,\xi+\beta^*_0)=\co^{n-1}_\ell(\t^*_0,\xi+\beta^*_0)\}
$$
and $\b^*_1$ an $\U$-valued and $\Fc_{\t^*_1}$-r.v. such that 
\begin{align*}
\co^{n-1}_{\t_1^*}(\t^*_0,\xi+\beta^*_0)&=
\E[\sum_{\t^*_1+1\le \ell < \t^*_1+\dl}e^{-\th \ell}g(X_\ell+\xi+\beta^*_0)1_{\{\ell \ge \t^*_0\}}|\Fc_{\t^*_1}]\\&\qquad \qquad +\max_{\beta \in \U}\E[-e^{-\th (\t^*_1+\dl)}\Psi(\beta)+Y^{n-2}_{\t^*_1+\dl}(\t^*_0,\xi+\beta^*_0+\b)|\Fc_{\t^*_1}]\\&=
\E[\sum_{\t^*_1+1\le \ell < \t^*_1+\dl}e^{-\th \ell}g(X_\ell+\xi+\beta^*_0)|\Fc_{\t^*_1}]\\&\qquad \qquad +\max_{\beta \in \U}\E[-e^{-\th (\t^*_1+\dl)}\Psi(\beta)+Y^{n-2}_{\t^*_1}(\t^*_1,\xi+\beta^*_0+\b)|\Fc_{\t^*_1}]\\&
=\E[\sum_{\t^*_1+1\le \ell < \t^*_1+\dl}e^{-\th \ell}g(X_\ell+\xi+\beta^*_0)|\Fc_{\t^*_1}]\\&\qquad \qquad +\E[-e^{-\th (\t^*_1+\dl)}\Psi(\beta^*_1)+Y^{n-2}_{\t^*_1+\dl}(\t^*_1,\xi+\beta^*_0+\beta^*_1)|\Fc_{\t^*_1}], \end{align*}
since
\begin{align*}
Y^{n-1}_{\t^*_0+\dl}(\t^*_0,\xi+\beta^*_0)&=
\E[\sum_{\t^*_0+\dl\le \ell\le\t_1^*}e^{-\th \ell}g(X_\ell+\xi+\beta^*_0)+\co^{n-1}_{\t_1^*}(\t^*_0,\xi+\beta^*_0)|\Fc_{\t^*_0+\dl}].
\end{align*}
Note that $\t_1^*$ is optimal after $\t_0^*+\dl$ since $\lim_{k\rw \infty}\co_k^{n-1}(\t_0^*,\xi+\beta_0^*)=\co_\infty^{n-1}(\t_0^*,\xi+\beta_0^*)=0$ (Remark \ref{domon}). 
Now, take the last expression of $\co^{n-1}_{\t_1^*}(\t^*_0,\xi+\beta^*_0)$ and insert it into the previous one of 
$Y^{n-1}_{\t^*_0+\dl}(\t^*_0,\xi+\beta^*_0)$
to obtain
\begin{align}\lb{eqxx}
Y^{n-1}_{\t^*_0+\dl}(\t^*_0,\xi+\beta^*_0)\nonumber&=
\E[\sum_{\t^*_0+\dl\le \ell<\t_1^*+\dl}e^{-\th \ell}g(X_\ell+\xi+\beta^*_0)-e^{-\th (\t^*_1+\dl)}\Psi(\beta^*_1) \\ & \qquad\qquad\qquad\qquad +Y^{n-2}_{\t^*_1+\dl}(\t^*_1,\xi +\beta^*_0+\beta^*_1)|\Fc_{\t^*_0+\dl}].
\end{align}
Next, in the last expression of $Y^n_k(\nu,\xi)$ in \eqref{eq7}, replace $Y^{n-1}_{\t^*_0+\dl}(\t^*_0,\xi+\beta^*_0)$ with the right-hand side of \eqref{eqxx} to obtain, for any $k\ge \nu$, 
\begin{align*}
Y_k^{n}(\nu,\xi)&=\E[\sum_{k\le \ell<\t_0^*+\dl}e^{-\th \ell}g(X_\ell+\xi)-e^{-\th (\t^*_0+\dl)}\Psi(\beta^*_0)+Y^{n-1}_{\t^*_0+\dl}(\t^*_0,\xi+\beta^*_0)|\Fc_k]\\&
=\E[\sum_{k\le \ell<\t_0^*+\dl}e^{-\th \ell}g(X_\ell+\xi)-e^{-\th (\t^*_0+\dl)}\Psi(\beta^*_0)+\sum_{\t^*_0+\dl\le \ell < \t^*_1+\dl}e^{-\th \ell}g(X_\ell+\xi+\beta^*_0)\\&\qq\qq-e^{-\th (\t^*_1+\dl)}\Psi(\beta^*_1)+Y^{n-2}_{\t^*_1+\dl}(\t^*_1,\xi+\beta^*_0+\beta^*_1)|\Fc_k]. 
\end{align*}
Continuing this procedure $n$ times, we deduce the existence of a strategy $\d^*=(\t^*_p,\b^*_p)_{p\ge 0}$ that belongs to $\ca_n^k$ such that, for any $k\ge \nu$,  
$$
Y^n_k(\nu,\xi)=
\E[\sum_{\ell\ge k}e^{-\th \ell}g(X^{\d^*}_\ell+\xi)1_{\{\ell \ge \nu\}}-\sum_{p \ge 0}e^{-\th (\t_p^*+\dl)}\Psi(\z^*_{p})|\Fc_k].
$$
On the other hand, the optimality of the choice of $(\t^*_p,\b^*_p)_{p\ge 0}$ implies that, for any $\d \in \ca_n^k$, we have
$$
Y^n_k(\nu,\xi)\ge 
\E[\sum_{\ell\ge k}e^{-\th \ell}g(X^{\d}_\ell+\xi)1_{\{\ell \ge \nu\}}-\sum_{p \ge 0}e^{-\th (\t_p+\dl)}\Psi(\z_{p})|\Fc_k],\quad k\ge \nu.
$$
Therefore, for any $n\ge 0$ and any pair $(\nu,\xi)$, 
$$
Y^n_k(\nu,\xi)=\esssup_{\d \in \ca^k_n}\E[\sum_{\ell\ge k}e^{-\th \ell}g(X^\d_\ell+\xi)1_{\{\ell \ge \nu\}}-\sum_{p \ge 0}e^{-\th (\t_p+\dl)}\Psi(\z_{p})|\Fc_k],\quad  k\ge \nu.
$$
\end{proof}
The strategy $\d_n^*:=(\t^*_p,\b^*_p)_{p\ge 0}$ displayed in the previous proof and which depends on $\nx$ is optimal in $\ca^\n_n$ when the system starts at $\nx$. Namely, we have the following
\begin{corollary}
\begin{align*}
    Y^n_\n(\nu,\xi)&=\esssup_{\d \in \ca_n^\n}\E[\sum_{\ell\ge \n}e^{-\th \ell}g(X^\d_\ell+\xi)1_{\{\ell \ge \nu\}}-\sum_{p \ge 0}e^{-\th (\t_p+\dl)}\Psi(\z_{p})|\Fc_\n]\\
&=
\E[\sum_{\ell\ge \nu}e^{-\th \ell}g(X^{\d^*_n}_\ell+\xi)1_{\{\ell \ge \nu\}}-\sum_{p \ge 0}e^{-\th (\t^*_p+\dl)}\Psi(\b_{p}^*)|\Fc_\n].
\end{align*}

\end{corollary}
\subsection{The general case}
We shall now consider the limit of $Y^n\nx$ as $n\to \infty$. 
We have the following
\begin{proposition}\label{conv} There exists a constant $C$ such that, for any pair 
$(\nu,\xi)$ and any $k\ge \nu$,
\begin{equation}\label{Y-n-k-bound}
|Y^n_k(\nu,\xi)|\le C,\qquad k\ge \nu.
\end{equation}
Consequently, the sequence of processes $(Y_k^n(\nu,\xi))_{k\ge \nu}$ converges $\P$-a.s. Moreover, the limit process 
$$
Y_k(\nu,\xi):=\lim_{n\rw \infty}Y^n_k(\nu,\xi),\quad k\ge \nu,
$$
satisfies
\begin{equation}\label{Y-k-nu}
Y_k(\nu,\xi)=\esssup_{\t\ge k}
\E[\sum_{k\le \ell\le\t}e^{-\th \ell}g(X_\ell+\xi)1_{\{\ell \ge \n\}}+\co_\t(\nu,\xi)|\Fc_k], \quad k\ge \nu,
\ee
where
\begin{equation}\label{obst-Y}
\co_k(\nu,\xi)=\left\{
\begin{array}{l}
\E[\underset{k+1\le \ell < k+\dl}{\sum}e^{-\th \ell}g(X_\ell+\xi)1_{\{\ell \ge \n\}}|\Fc_k] \\\qq\qq\q+\max_{\beta \in \U}\E[-e^{-\th (k+\dl)}\Psi(\beta)+Y_{k+\dl}(\n,\xi+\beta)|\Fc_k],\,\,\n\le k<\infty;\\
0 \mbox{ if }k=\infty. 
\end{array}
\right.
\end{equation}
Furthermore, 
\be \lb{Y-Snell}
Y_\infty^-(\nu,\xi)=\lim_{k\rwi}Y_k(\nu,\xi)=0=Y_\infty(\nu,\xi).
\ee
\end{proposition}
\begin{proof} We first derive \eqref{Y-n-k-bound}. Recall that 
$$
Y^n_k(\nu,\xi)=\esssup_{\d \in \ca^k_n}\E[\sum_{\ell\ge k}e^{-\th \ell}g(X^\d_\ell+\xi)1_{\{\ell \ge \nu\}}-\sum_{p \ge 0}e^{-\th (\t_p+\dl)}\Psi(\z_{p})|\Fc_k],\qquad  k\ge \nu.
$$
For any $\d=(\t_p,\xi_p)_{p\ge 0}\in \ca^k_n$, we have 
$\t_0\ge k$, $\t_n=\infty$ and $\t_{p+1}\ge \t_p+\Dl$. Therefore, 
$\t_p\ge k+p\Dl$ for any $p\ge 0$. So, 
$$
|-\sum_{p \ge 0}e^{-\th (\t_p+\dl)}\Psi(\z_{p})|\le e^{-\Dl \theta}\sum_{p=0}^{n-1}e^{-\theta (k+p\Dl)}\|\Psi\|_\infty\le 
e^{-(\Dl+k) \theta}(1-e^{-\theta \Dl})^{-1}\|\Psi\|_\infty.
$$
Since $g(\cdot)$ is bounded, the above characterization of $Y^n_k\nx$ implies that, for any $k\ge \nu$, 
\be\lb{estykn}
|Y^n_k(\nu,\xi)|\le 
e^{-k \theta}(1-e^{-\theta })^{-1}\|g\|_\infty+
e^{-(\Dl+k) \theta}(1-e^{-\theta \Dl})^{-1}\|\Psi\|_\infty.\ee
So, let us take a constant $C$ to be 
$$C:=(1-e^{-\theta })^{-1}\|g\|_\infty+
e^{-\Dl\theta}(1-e^{-\theta \Dl})^{-1}\|\Psi\|_\infty.$$
Then, for any $k\ge \nu$, $|Y^n_k(\nu,\xi)|\le C$, $\P$-a.s. Next, since the sequence of processes $(Y^n_{k}(\xi,\nu))_{k\ge \nu}$ is increasing and bounded by $C$, it converges $\P$-a.s. to a bounded process $(Y_{k}(\xi,\nu))_{k\ge \nu}$ which, due to \eqref{estykn}, satisfies 
\be\lb{estyinf3}
|Y_{k}(\xi,\nu)|\le e^{-k \theta} C,\qquad k\ge \nu.
\ee
Therefore, 
$$
\lim_{k\rwi}Y_k(\nu,\xi)=0
=Y_\infty(\nu,\xi).
$$
Next, since $\U$ is finite, by \eqref{obst} and conditional dominated convergence we have $\P$-a.s.
$$
\underset{n\to\infty}{\lim}\co^n_k(\nu,\xi)=\co_k(\nu,\xi),\quad k\ge \nu,
$$
where, for $k\ge \n$, 
$$
\begin{array}{lll}
\co_k(\nu,\xi)=\E[\underset{k+1\le \ell < k+\dl}{\sum}e^{-\th \ell}g(X_\ell+\xi)1_{\{\ell \ge \n\}}|\Fc_k] \\ \qquad\qquad\qquad\qquad\qquad\qquad +\underset{\beta \in \U}{\max}\, \E[-e^{-\th (k+\dl)}\Psi(\beta)+Y_{k+\dl}(\n,\xi+\beta)|\Fc_k].
\end{array}
$$
Recall the Snell envelope  $(\widetilde{Y}_k^n(\nu,\xi))_{k\ge 0}$ defined by \eqref{Y-tilde-n} and associated with the obstacle process $(L^n_k\nx)_{k\ge 0}$ given by \eqref{L-k-n}. Since 
$g(\cdot)$ is bounded, the Snell envelope $(\widetilde{Y}_k^n(\nu,\xi))_{k\ge \nu}$ converges $\P$-a.s. to the process $(\widetilde{Y}_k(\nu,\xi))_{k\ge \nu}$ defined, for any $(\nu,\xi)$, by
\be\lb{Y-tilde}
\widetilde{Y}_k(\xi,\nu):=Y_k(\nu,\xi)+\sum_{\ell \le k-1}e^{-\th \ell}g(X_\ell+\xi)1_{\{\ell \ge \n\}},\quad k\ge \nu.
\ee
By \eqref{eq2-tilde-n} and conditional dominated convergence, we have $\P$-a.s.
\be \lb{eq2-tilde-nxx}\left\{\begin{array}{l}
\widetilde{Y}_k(\nu,\xi)=\max\{\E[\widetilde{Y}_{k+1}(\nu,\xi)|\Fc_k], \underbrace{\sum_{\ell\le k}e^{-\th \ell}g(X_\ell+\xi)1_{\{\ell \ge \n\}}+\co_k(\nu,\xi)}_{L_k\nx}\},\,\,  k\ge \n.
\\
\widetilde{Y}_\infty(\nu,\xi)=L_\infty\nx=
\sum_{\ell\ge 0}e^{-\th \ell}g(X_\ell+\xi)1_{\{\ell \ge \n\}}.
\end{array}\right.
\ee
Thus, $(\widetilde{Y}_{k}(\xi,\nu))_{k\ge \nu}$ is the Snell envelope associated with the obstacle process  $(L_k\nx)_{k\ge \n}$. Therefore, for any $k\ge \nu$,
$$
Y_k(\nu,\xi)+\sum_{\ell \le k-1}e^{-\th \ell}g(X_\ell+\xi)1_{\{\ell \ge \n\}}=\esssup_{\t\ge k}
\E[\sum_{\ell\le\t}e^{-\th \ell}g(X_\ell+\xi)1_{\{\ell \ge \n\}}+\co_\t(\nu,\xi)|\Fc_k],
$$
which in turn yields \eqref{Y-k-nu}.
\end{proof}
As a consequence of Proposition \ref{cs}, we also have the following consistency property for the limit processes $(Y_k\nx)_{k\ge \n}$.
\begin{corollary}
\label{csy}
Let $\nu$ be a stopping time and $\xi$ a $\U$-valued and $\Fc_{\nu}$-measurable r.v. Then, for any stopping time $\nu'\ge \nu$, $\P$-a.s, it holds that, 
$$
Y_k(\nu,\xi)=Y_k(\nu',\xi), \quad k\ge \nu'.
$$
\end{corollary}
\begin{remark} Note that the estimate \eqref{Y-n-k-bound} is valid only for $k\ge \n$. This estimate is enough to solve our impulse control problem in this discrete time setting with delay. However, when $\,\Psi\ge 0$, one can show easily by induction on $n$ that the process $(Y_k^{n}(\nu,\xi))_{k\ge 0}$ satisfies the following estimate: for any $k\ge 0$,
$$|Y_k^{n}(\nu,\xi)|\le e^{-k\th}(1-\th)^{-1}\|g\|_\infty, $$
which  means that under the assumption $\Psi\ge 0$, the estimate \eqref{Y-n-k-bound} is valid not only for $k\ge \n$ but also for $k<\nu$. Finally, when the sign of $\Psi$ is negative, one can also obtain a similar estimate for $Y_k^{n}(\nu,\xi)$ but the r.h.s. contains $\|\Psi\|_\infty$ (see Section \ref{risk-s} below for more details).  
\end{remark}

The following is the main result of this section.
\begin{theorem}\lb{opt-strat-rn} The discrete time impulse control with delay problem admits an optimal strategy $\d^*=(\t_n^*,\xi_n^*)_{n\ge 0}$. Moreover,
$$
Y_0(0,0):=Y_0(\n,\xi)_{|(\n,\xi)=(0,0)}=\sup_{\d \in \ca}J(\d)=J(\d^*).
$$
\end{theorem}
\def \os {\d^*=(\t_n^*,\xi_n^*)_{n\ge 0}}
\begin{proof}
\nd  First, consider the following strategy $\d^*=(\t_n^*,\xi_n^*)_{n\ge 0}$ defined as follows.

\nd (i) $$
\t^*_0=\inf\{\ell \ge \Delta,\,\, 
Y_\ell(0,0)=\co_\ell(0,0)\}.
$$
(ii) $\xi_0^*$ is the $\U$-valued and $\Fc_{\t_0^*}$-measurable r.v. such that 
\begin{align}\lb{eqsimplifo}
\max_{\beta \in \U}\E[-e^{-\th (\t^*_0+\dl)}\Psi(\beta)+Y_{\t^*_0+\dl}(0,\beta)|\Fc_{\t^*_0}]
&=\max_{\beta \in \U}\E[-e^{-\th (\t^*_0+\dl)}\Psi(\beta)+Y_{\t^*_0+\dl}(\t^*_0,\beta)|\Fc_{\t^*_0}]\nn\\&=
\E[-e^{-\th (\t^*_0+\dl)}\Psi(\xi_0^*)+Y_{\t^*_0+\dl}(\t^*_0,\xi_0^*)|\Fc_{\t^*_0}].
\end{align}

\nd(iii) For $n\ge 1$,
$$
\t^*_n=\inf\{\ell \ge \Delta+\t^*_{n-1},\,\, 
Y_\ell(\t^*_{n-1},\xi_0^*+\ldots+\xi_{n-1}^*)=\co_\ell(\t^*_{n-1},\xi_0^*+\ldots+\xi_{n-1}^*)\}
$$
(iv) $\xi_n^*$ is the $\U$-valued and $\Fc_{\t_0^*}$-measurable r.v. such that
\begin{align}
&\max_{\beta \in \U}\E[-e^{-\th (\t^*_n+\dl)}\Psi(\beta)+Y_{\t^*_n+\dl}(\t^*_n,\xi_0^*+\ldots+\xi_{n-1}^*+\beta)|\Fc_{\t^*_n}]=
\nn \\&\qquad\qquad \qquad \qquad \qquad \E[-e^{-\th (\t^*_n+\dl)}\Psi(\xi_n^*)+Y_{\t^*_n+\dl}(\t^*_n,\xi_0^*+\ldots+\xi_{n-1}^*+\xi_n^*)|\Fc_{\t^*_n}].
\end{align}
\def \D {\Delta}
The strategy $\d^*$ is admissible since, for any $n\ge 0$, $\t^*_n$ is a stopping time, $\t^*_n-\t^*_{n-1}\ge \D$ and $\xi_n^*$ is a $\U$-valued r.v. which is moreover $\Fc_{\t^*_n}$-measurable. We will next show that
$$
Y_0(0,0)=J(\d^*).
$$
Indeed, recall that $(Y_k(\nu,\xi))_{k\ge \nu}$ satisfies
\begin{equation}\label{Y-k-nu1xx}
Y_k(\nu,\xi)=\esssup_{\t\ge k}
\E[\sum_{k\le \ell\le\t}e^{-\th \ell}g(X_\ell+\xi)1_{\{\ell \ge \n\}}+\co_\t(\nu,\xi)|\Fc_k], \quad k\ge \nu.
\ee
Therefore, by taking $\nu=0$ and $\xi=0$, which corresponds to the initial state of the system, it holds that
\begin{equation}\label{Y-k-u}
Y_k(0,0)=\esssup_{\t\ge k}
\E[\sum_{k\le \ell\le\t}e^{-\th \ell}g(X_\ell)+\co_\t(0,0)|\Fc_k], \quad k\ge 0.
\ee
Since
$$
\t^*_0=\inf\{\ell \ge 0,\,\, 
Y_\ell(0,0)=\co_\ell(0,0)\},
$$
it is optimal for \eqref{Y-k-u} with $k=0$ simply  because
$\underset{k\rwi}{\lim}\,\co_k(0,0)=\co_\infty(0,0)=0$ and $g(\cdot)$ is bounded. So, $\t^*_0$ satisfies
\begin{align}
\lb {eq7x}
Y_0(0,0)& =
\E[\sum_{0\le \ell\le\t_0^*}e^{-\th \ell}g(X_\ell)+\co_{\t_0^*}(0,0)].
\end{align}
But,
$$\begin{array}{ll}
\co_{\t_0^*}(0,0)
&=\E[\underset{\t^*_0+1\le \ell < \t^*_0+\dl}{\sum}\,e^{-\th \ell}g(X_\ell)|\Fc_{\t^*_0}] +\underset{\beta \in \U}{\max}\,\E[-e^{-\th (\t^*_0+\dl)}\Psi(\beta)+Y_{\t^*_0+\dl}(0,\beta)|\Fc_{\t^*_0}]\\
&=\E[\underset{\t^*_0+1\le \ell < \t^*_0+\dl}{\sum}\,e^{-\th \ell}g(X_\ell)|\Fc_{\t^*_0}] +\underset{\beta \in \U}{\max}\,\E[-e^{-\th (\t^*_0+\dl)}\Psi(\beta)+Y_{\t^*_0+\dl}(\t^*_0,\beta)|\Fc_{\t^*_0}]\\
& =\E[\underset{\t^*_0+1\le \ell < \t^*_0+\dl}{\sum}\,e^{-\th \ell}g(X_\ell)|\Fc_{\t^*_0}]+ \E[-e^{-\th (\t^*_0+\dl)}\Psi(\xi_0^*)+Y_{\t^*_0+\dl}(\t^*_0,\xi_0^*)|\Fc_{\t^*_0}], \end{array}
$$
where the second equality is due to the consistency property of Corollary \ref{csy} since $\b\in \U$ is deterministic and then $\Fc_{\t^*_0}$-measurable. Insert the last expression of $\co_{\t_0^*}(0,0)$ in the r.h.s. of \eqref{eq7x} to obtain
\begin{align}
Y_0(0,0) =
\E[\sum_{0\le \ell<\t_0^*+\Dl}e^{-\th \ell}g(X_\ell)-e^{-\th (\t^*_0+\dl)}\Psi(\xi_0^*)+Y_{\t^*_0+\dl}(\t^*_0,\xi_0^*)].
\end{align}
Next, we have 
$$Y_{\t^*_0+\dl}(\t^*_0,\xi_0^*)=\E[\sum_{{\t^*_0+\dl}\le \ell\le\t_1^*}e^{-\th \ell}g(X_\ell+\xi_0^*)+\co_{\t_1^*}(\t^*_0,\xi_0^*)|\Fc_{\t^*_0+\dl}]
$$since $\t_1^*$ is optimal after 
$\t^*_0+\dl$. On the other hand,
\begin{align}
\co_{\t_1^*}(\t^*_0,\xi_0^*)=&\E[\underset{{\t_1^*}+1\le \ell < {\t_1^*}+\dl}{\sum}e^{-\th \ell}g(X_\ell+\xi_0^*)|\Fc_{\t_1^*}]\nn \\ &\qquad\qquad\qquad\qquad +\max_{\beta \in \U}\E[-e^{-\th ({\t_1^*}+\dl)}\Psi(\beta)+Y_{{\t_1^*}+\dl}(\t_0^*,\xi_0^*+\beta)|\Fc_{\t_1^*}]\nn\\
=&\E[\underset{{\t_1^*}+1\le \ell < {\t_1^*}+\dl}{\sum}e^{-\th \ell}g(X_\ell+\xi_0^*)|\Fc_{\t_1^*}] \nn \\ &\qquad\qquad\qquad\qquad +\max_{\beta \in \U}\E[-e^{-\th ({\t_1^*}+\dl)}\Psi(\beta)+Y_{{\t_1^*}+\dl}(\t_1^*,\xi_0^*+\beta)|\Fc_{\t_1^*}]\lb{ot1}\\
=&\E[\underset{{\t_1^*}+1\le \ell < {\t_1^*}+\dl}{\sum}e^{-\th \ell}g(X_\ell+\xi_0^*)|\Fc_{\t_1^*}] \nn \\ &\qquad\qquad\qquad\qquad +\E[-e^{-\th ({\t_1^*}+\dl)}\Psi(\xi_1^*)+Y_{{\t_1^*}+\dl}(\t_1^*,\xi_0^*+\xi_1^*)|\Fc_{\t_1^*}].\nn 
\end{align}
By inserting the last expression of \eqref{ot1} in the ones of $Y_{\t^*_0+\dl}(\t^*_0,\xi_0^*)$ and $Y_0(0,0)$ successively, we obtain
\begin{align*}
Y_0(0,0)=&\E[\sum_{0\le \ell<\t_0^*+\Dl}e^{-\th \ell}g(X_\ell)-e^{-\th (\t^*_0+\dl)}\Psi(\xi_0^*)+\sum_{{\t^*_0+\dl}\le \ell<\t_1^*+\dl}e^{-\th \ell}g(X_\ell+\xi_0^*)+\\ & \qquad -e^{-\th ({\t_1^*}+\dl)}\Psi(\xi_1^*)+Y_{{\t_1^*}+\dl}(\t_1^*,\xi_0^*+\xi_1^*)].
\end{align*}
By repeating this procedure $n$ times yields
\begin{align*}
Y_0(0,0)=&\E[\sum_{0\le \ell<{\t^*_n+\dl}}e^{-\th \ell}g(X^{\d^*}_\ell)-\sum_{0\le \ell\le n}\Psi (\xi^*_\ell)e^{-\th (\t^*_\ell+\dl)}+Y_{\t^*_n+\dl}(\t^*_{n},\xi_0^*+...+\xi_{n}^*)].
\end{align*}
We have 
$$
|Y_0(0,0)-J(\d^*)|\le \E[|\sum_{\t^*_n+\dl\le \ell}e^{-\th \ell}g(X^{\d^*}_\ell)|+|\sum_{n+1\le \ell}\Psi (\xi^*_\ell)e^{-\th (\t^*_\ell+\dl)}|+|Y_{\t^*_n+\dl}(\t^*_{n},\xi_0^*+...+\xi_{n}^*)|].
$$
Since $\t_n^*\ge n\dl$, $g(\cdot)$ and $\Psi(\cdot)$ are bounded, due to \eqref{estyinf3} the r.h.s. of the above inequality converges to $0$, as $n$ goes to infinity, which means that $Y_0(0,0)=J(\d^*)$.

Let $\d=(\t_n,\z_n)_{n\ge 0}$ be an admissible strategy. Since $\t_0$ is an arbitrary stopping time, then 
\begin{equation}\label{Y-k-nu2}
Y_0(0,0)\ge \E[\sum_{0\le \ell\le\t_0}e^{-\th \ell}g(X_\ell)+\co_{\t_0}(0,0)]. 
\ee
But, 
$$\begin{array}{ll}\co_{\t_0}(0,0)&=
\E[\underset{\t_0+1\le \ell < \t_0+\dl}{\sum}e^{-\th \ell}g(X_\ell)|\Fc_{\t_0}] +\max_{\beta \in \U}\E[-e^{-\th (\t_0+\dl)}\Psi(\beta)+Y_{\t_0+\dl}(0,\beta)|\Fc_{\t_0}]\\
&\ge 
\E[\underset{\t_0+1\le \ell < \t_0+\dl}{\sum}e^{-\th \ell}g(X_\ell)|\Fc_{\t_0}] +\E[-e^{-\th (\t_0+\dl)}\Psi(\z_0)+Y_{\t_0+\dl}(\t_0,\z_0)|\Fc_{\t_0}].\\
\end{array}
$$
Therefore, we obtain from \eqref{Y-k-nu2} that
\begin{equation}\label{Y-k-nu3}
Y_0(0,0)\ge \E[\sum_{0\le \ell<\t_0+\dl}e^{-\th \ell}g(X_\ell)-e^{-\th (\t_0+\dl)}\Psi(\z_0)+Y_{\t_0+\dl}(\t_0,\z_0)]. 
\ee
By noting that
\begin{equation}\label{Y-k-nu1}
Y_{\t_0+\dl}(\t_0,\z_0)\ge \E[\sum_{\t_0+\dl\le \ell\le\t_1}e^{-\th \ell}g(X_\ell+\z_0)+\co_{\t_1}(\t_0,\z_0)|\Fc_{\t_0+\dl}]
\ee
and 
$$\begin{array}{ll}
\co_{\t_1}(\t_0,\z_0)=
\E[\underset{\t_1+1\le \ell < \t_1+\dl}{\sum}e^{-\th \ell}g(X_\ell+\z_0)|\Fc_{\t_1}] +\max_{\beta \in \U}\E[-e^{-\th (\t_1+\dl)}\Psi(\beta)+Y_{\t_1+\dl}(\t_0,\z_0+\beta)|\Fc_{\t_1}]\\ \qquad\qquad\quad
=
\E[\underset{\t_1+1\le \ell < \t_1+\dl}{\sum}e^{-\th \ell}g(X_\ell+\z_0)|\Fc_{\t_1}] +\max_{\beta \in \U}\E[-e^{-\th (\t_1+\dl)}\Psi(\beta)+Y_{\t_1+\dl}(\t_1,\z_0+\beta)|\Fc_{\t_1}]\\  \qquad\qquad\quad
\ge 
\E[\underset{\t_1+1\le \ell < \t_1+\dl}{\sum}e^{-\th \ell}g(X_\ell+\z_0)|\Fc_{\t_1}] +\E[-e^{-\th (\t_1+\dl)}\Psi(\z_1)+Y_{\t_1+\dl}(\t_1,\z_0+\z_1)|\Fc_{\t_1}],
\end{array}
$$
we get from \eqref{Y-k-nu3}, 
\begin{align*}
Y_0(0,0)&\ge 
\E[\sum_{0\le \ell<\t_0+\dl}e^{-\th \ell}g(X_\ell)+\sum_{\t_0+\dl\le \ell<\t_1+\dl}e^{-\th \ell}g(X_\ell+\z_0)-e^{-\th (\t_0+\dl)}\Psi(\z_0)+\\&
\qq \qq -e^{-\th (\t_1+\dl)}\Psi(\z_1)+Y_{\t_1+\dl}(\t_1,\z_0+\z_1)]\\&=
\E[\sum_{0\le k<\t_1+\dl}e^{-\th k}g(X^\d_k)-\sum_{\ell =0}^1e^{-\th (\ell+\dl)}\Psi(\z_{\ell})+Y_{\t_1+\dl}(\t_1,\z_0+\z_1)].\end{align*}
By repeating this procedure as many times as necessary we get, for any $n\ge 0$,
\begin{align*}
Y_0(0,0)\ge
\E[\sum_{0\le k<\t_n+\dl}e^{-\th k}g(X^\d_k)-\sum_{\ell =0}^n e^{-\th (\ell+\dl)}\Psi(\z_{\ell})+Y_{\t_n+\dl}(\t_n,\z_0+\z_1+\ldots+\z_n)].\end{align*}
Due to the boundedness of $g(\cdot)$ and $\Psi(\cdot)$, the estimate \eqref{estyinf3} and the fact that $\t_n\ge n\dl$, by taking the limit $n\to \infty$ in the above inequality, we finally obtain
$$
Y_0(0,0)=J(\d^*)\ge J(\d),
$$ which ends the proof of the claim, since 
$\d$ is arbitrary in $\ca$.
\end{proof}
\subsection{Bounded $\eps$-optimal strategies} We will show that, for every $\eps>0$, there exists a strategy $\d_\eps$ which allows at most a bounded (i.e., finite and independent of $\omega$) number $n_\eps$ of impulses  such that 
$\sup_{\d\in \ca}J(\d)\le J(\d_\eps)+\eps.$ This property is  very useful in settings where exact optimal strategies may not be easily implementable.

\begin{proposition} Let $\eps>0$. Then there exists $n_\eps\ge 1$ and $\d_\eps\in\ca_{n_\eps}$ such that 
$$\sup_{\d\in \ca}J(\d)\le J(\d_\eps)+\eps,$$
that is, the strategy $\d_{\eps}$ is $\eps$-optimal. 
\end{proposition}
\begin{proof}
Let $\eps>0$ and $n_\eps$ such that 
\def \ng{\|g\|_\infty}
\def \np{\|\Psi\|_\infty}
\be \lb{tronc}
\sum_{\ell \ge (1+n_\eps)\dl}e^{-\ell\dl}\ng +
\sum_{\ell \ge n_\eps}e^{-(\ell+1)\dl}\np<\eps.
\ee
i.e.,
$$
n_{\eps}\ge \lfloor(\dfrac{1}{\dl}\ln{(\dfrac{C_0}{\eps})}\rfloor+1,
$$
where $C_0:=C_0(g,\psi,\Delta)=\dfrac{e^{-\dl}}{1-e^{-\dl}} (\|g\|_{\infty}+\|\Psi\|_{\infty})$.

We have
\begin{align*}
\sup_{\d \in \ca}J(\d)=J(\d^*)&=
\E[\sum_{0\le \ell<{\t^*_{n_\eps}+\dl}}e^{-\th \ell}g(X^{\d^*}_\ell)-\sum_{0\le \ell\le n_\eps}\Psi (\xi^*_\ell)e^{-\th (\t^*_\ell+\dl)}]\\&\qquad\qquad +\E[\sum_{ \ell\ge {\t^*_{n_\eps}+\dl}}e^{-\th \ell}g(X^{\d^*}_\ell)-\sum_{ \ell\ge 1+n_\eps}\Psi (\xi^*_\ell)e^{-\th (\t^*_\ell+\dl)}]\\
&\le  
\E[\sum_{0\le \ell<{\t^*_{n_\eps}+\dl}}e^{-\th \ell}g(X^{\d^*}_\ell)-\sum_{0\le \ell\le n_\eps}\Psi (\xi^*_\ell)e^{-\th (\t^*_\ell+\dl)}]+\eps,
\end{align*}
since, due to \eqref{tronc}, the remaining terms are smaller than $\eps$. 
But, 
$$
\E[\sum_{0\le \ell<{\t^*_{n_\eps}+\dl}}e^{-\th \ell}g(X^{\d^*}_\ell)-\sum_{0\le \ell\le n_\eps}\Psi (\xi^*_\ell)e^{-\th (\t^*_\ell+\dl)}]\le \sup_{\d \in \ca_{n_\eps}}J(\d)=J(\d_\eps)$$
where $\d_\eps$ is an optimal strategy in $\ca_{n_\eps}$.  Therefore, 
$$\sup_{\d\in \ca}J(\d)\le J(\d_\eps)+\eps,
$$
i.e., the strategy $\d_{\eps}$ is $\eps$-optimal.
\end{proof}
\section{The discrete time risk-sensitive impulse control problem}\lb{risk-s}
In this section, we extend the previous results to the risk-sensitive case where the controller has a utility function of exponential type. We tackle this problem  using a probabilistic approach based on the notion of  the Snell envelope of processes. When the decision maker implements a strategy $\delta=(\tau_n,\xi_n)_{n\ge 0}$, the payoff is given by
\def \r{\rho}
\def \th {\theta}
\begin{equation}\label{reward2}
J(\delta):=\mathbb{E} [\exp\{\r\{\underbrace{\sum_{\ell \ge 0}e^{-\th \t_\ell}g(X^\d_\ell)-\sum_{\ell\geq 0}e^{-\th(\tau_{\t_\ell}+\dl )}\Psi(\xi_{\ell})}_{C(\d)}\}], 
\end{equation}
where $\r>0$ is the risk-sensitive index and $X^\d$ is the dynamics of the controlled system given by 
\begin{equation}\label{X-delta-r}
X^\d_k=X_k 1_{\{k<\t_0+\dl\}}+\sum_{\ell \ge 0}(X_k+\z_0+\dots+ \z_\ell) 1_{\{\t_\ell+\dl \le k<\t_{\ell+1}+\dl\}},\quad k\ge 0,
\end{equation}
which is \eqref{X-delta}.

Exponential utilities are often called risk-sensitive utilities because they capture risk-averse and risk-seeking behaviors of the controller (see \cite{jacobson}). 
To see this, note that if the index $\r$
is small enough then the risk-sensitive reward function satisfies
\begin{equation}\label{justrisksensitive}
\Gamma(\r,\delta):=\r^{-1}\mbox{Log}(J(\delta))\approx \mathbb{E}[C(\delta)]+\frac{\r}{2}\mbox{var}[C(\delta)],
\end{equation}
where $\mbox{var}[C(\delta)]$ denotes the variance of $C(\d)$. 
Thus, $\lim_{\r \rightarrow 0}\,
\Gamma(\theta,\delta)=\mathbb{E}[C(\delta)]
$ which is the risk-neutral utility function \eqref{rn} studied in Section \ref{rn-case}. 
Now, if $\r > 0$, we have $\Gamma(\r,\delta)>\mathbb{E}[C(\delta)]$ meaning that the variance $\mbox{var}[C(\delta)]$ as a measure of risk improves $\Gamma(\r,\delta)$, in which case the optimizer is a risk-seeker. But, when $\r<0$,  $\Gamma(\r,\delta)<\mathbb{E}[C(\delta)]$ meaning that the variance $\mbox{var}[C(\delta)]$ worsens $\Gamma(\r,\delta)$ in which case the optimizer is risk-averse.

For simplicity, we hereafter consider only the case $\r=1$ since the other cases are treated in a similar way. We proceed by recasting the risk-sensitive impulse control problem into an iterative optimal stopping problem, and by exploiting once more the properties of the Snell envelope, we shall be able to characterize recursively an optimal strategy to this discrete time risk-sensitive impulse control problem. 

\subsection{Iterative optimal stopping and properties}
 Let  $\nu$ be a stopping time and $\xi$ an $\mathcal{F}_{\nu}$-measurable random variable. Consider the sequence of processes  $(V^{n}(\nu, \xi))_{n\geq 0}$  defined recursively by
\begin{equation}\label{Y0app}
 V_{k}^{0}(\nu,\xi)= \mathbb{E}\left[\exp\left\lbrace \sum_{\ell\ge k}e^{-\th\ell }g(X_k+\xi)1_{\{\n\ge \ell\}}\right\rbrace|\mathcal{F}_{k}\right],\, k\in \nb, \,\,  \text{ and }V_{\infty}^{0}=V_{\infty}^{0,-}=1.
\end{equation}
For $n\geq 1$, 
\begin{equation} \label{Yapp}
V_{k}^{n}(\nu,\xi)= \esssup\limits_{\tau\ge k}\mathbb{E}\left[\exp\left\lbrace \sum_{k \le \ell \le \t}e^{-\th \ell}g(X_\ell+\xi)1_{\{\ell\ge \n\}}\right\rbrace \Theta_{\tau}^{n}(\nu,\xi) \,|\, \mathcal{F}_{k}\right],\,\, k\in \bbn,
\end{equation}
where for $k\in \nb$, 
\begin{eqnarray*}
\Theta_{k}^{n}(\nu,\xi)= \max\limits_{\beta \in U} \bigg\lbrace  \mathbb{E}\bigg\lbrack \exp\bigg\lbrace \sum_{k+1\le \ell <k+\dl}e^{-\th \ell}g(X_\ell+\xi)1_{\{\ell \ge \n\}}-e^{-\th(k+\Delta)}\Psi(\beta) \bigg\rbrace V_{k+\Delta}^{n-1}(\nu,\xi+\beta)|\mathcal{F}_{k}\bigg\rbrack \bigg\rbrace.
\end{eqnarray*}and 
$\Theta_{\infty}^{n}(\nu,\xi)=1.$ \\\\ 
In the next proposition, we collect some properties of the sequence of processes 
$(V_k^{n}(\nu,\xi))_{k\in \bbn}$, $\,n\ge 0$.
\begin{proposition} The following properties hold true. 
\begin{itemize}
\item[(a)] For any $n\ge 0$ and any pair $(\n,\xi)$, the processes $(V_k^{n}(\nu,\xi))_{k\in \bbn}$ are well-posed. 
\item[(b)](Monotonicity). For any $\nx$, the sequence of processes $((V_k^{n}(\nu,\xi))_{k\in \bbn})_{n\ge 0}$ is non-decreasing in $n$, i.e., for any $n\ge0$ and any pair $(\nu,\xi)$, 
\def \lb{\label}
\begin{equation} \label{monotonicity-rs}
0\le V^n_k(\nu,\xi)\le V^{n+1}_k(\nu,\xi),\quad k\ge 0.
\end{equation}
\item[(c)] (Consistency). Let $\nu$ be a stopping time and $\xi$ be a $\U$-valued and $\Fc_{\nu}$-measurable r.v. Then, for any $\F$-stopping time $\nu'\ge \nu$, $\P$-a.s, it holds that, for any $n\ge 0$, 
\begin{align}\lb {consiscdt}V^{n}_k(\nu,\xi)=V^{n}_k(\nu',\xi), \quad \,\,k\ge \nu'.\end{align}
\end{itemize}
\end{proposition}
\begin{proof} (a) For $n=0$ and any $\nx$, the process $(V_k^{0}(\nu,\xi))_{k\in \bbn}$ is well posed, since $g$ is bounded. Moreover $$|V_k^{0}(\nu,\xi)|\le \exp{\{(1-e^{-\th})^{-1}\|g\|_\infty e^{-\th k}\}}, \quad  k\in \bbn.$$
Next, the boundedness of $g$, $\psi$ and $(V^0_k\nx)_{k\in \bbn}$ uniformly in $\nx$, implies that 
$$\E[\sup_{k\in \bbn}|\Theta_{k}^{1}(\nu,\xi)|]<\infty.
$$
So, $(V^1_k\nx)_{k\in \nb}$ is well-posed. Furthermore, $V^1_\infty\nx=\Theta^1_\infty\nx=1.$ Finally, for any $k\in \bbn$, 
\def \ngf {\|g\|_{\infty}}
\def \cth {(1-e^{-\t})^{-1}}
\begin{align*}\begin{array}{ll}
|\Theta_{k}^{1}(\nu,\xi)|&\le \exp\{\|g\|_{\infty}(1-e^{-\th})^{-1}(e^{-\th (k+1)}-e^{-\th (k+\dl)})+e^{-\th (k+\dl)}\|\Psi\|_\infty\} \exp{\{(1-e^{-\th})^{-1}\|g\|_\infty e^{-\th (k+\dl)}\}}\\ \\
&\le \exp\{\|g\|_{\infty}(1-e^{-\th})^{-1}e^{-\th (k+1)}+e^{-\th (k+\dl)}\|\Psi\|_\infty\}.
\end{array}
\end{align*}
\def \qq {\qquad}
Hence,  for any $k\in \bbn$,
$$
|V_k^{1}(\nu,\xi)|\le 
\exp\{\|g\|_{\infty}(1-e^{-\th})^{-1}e^{-\th k}+e^{-\th (k+\dl)}\|\Psi\|_\infty\}.
$$
Assume now that, for some $n$ and for any pair $\nx$, the process 
$(V^n_k(\nx))_{k\in \bbn}$ is well-posed (
$V^n_\infty(\nx)=1$) and for any $k\ge 0$,  
$$
|V_k^{n}(\nu,\xi)|\le 
\exp\{\|g\|_{\infty}(1-e^{-\th})^{-1}e^{-\th k}+ \|\Psi\|_\infty\sum_{j=1}^ne^{-\th (k+j\dl)}\}.
$$
Therefore,
\begin{align*}
|\Theta_{k}^{n+1}(\nu,\xi)|&\le \max\limits_{\beta \in U} \bigg\lbrace  \mathbb{E}\bigg\lbrack \exp\bigg\lbrace \sum_{k+1\le \ell <k+\dl}e^{-\th \ell}g(X_\ell+\xi)1_{\{\ell \ge \n\}}-e^{-\th(k+\Delta)}\Psi(\beta) \bigg\rbrace |V_{k+\Delta}^{n}(\nu,\xi+\beta)||\mathcal{F}_{k}\bigg\rbrack \bigg\rbrace\\
&\le \exp\{\|g\|_{\infty}(1-e^{-\th})^{-1}(e^{-\th (k+1)}-e^{-\th (k+\dl)})+e^{-\th (k+\dl)}\|\Psi\|_\infty\}\\&\qq\qq\times \exp\{\|g\|_{\infty}(1-e^{-\th})^{-1}e^{-\th (k+\dl)}+  \|\Psi\|_\infty \sum_{j=1}^ne^{-\th(k+\dl+j\dl)}\}\\&
=\exp\{\|g\|_{\infty}(1-e^{-\th})^{-1}e^{-\th (k+1)}+  \|\Psi\|_\infty \sum_{j=1}^{n+1}e^{-\th(k+j\dl)}\},
\end{align*}
which implies that $(V_{k}^{n+1}(\nu,\xi))_{k\in \bbn}$ is well-posed, $V_{\infty}^{n+1}(\nu,\xi)=1$, and for any $k\ge 0$,
\begin{align*}
|V_{k}^{n+1}(\nu,\xi)|\le 
\exp\{\|g\|_{\infty}(1-e^{-\th})^{-1}e^{-\th k}+ \|\Psi\|_\infty\sum_{j=1}^{n+1}e^{-\th (k+j\dl)}\}.
\end{align*}
Consequently, for any $n\ge 0$ and any pair $\nx$, $(V_k^n\nx)_{k\in \bbn}$ is well-posed and the following estimate holds true, for any $n\ge 0$, any pair $\nx$ and and $k\in \bbn$:
\begin{align}\lb{estvn}
|V_k^{n}(\nu,\xi)|\le 
\exp\{\|g\|_{\infty}(1-e^{-\th})^{-1}e^{-\th k}+ \|\Psi\|_\infty\sum_{j=1}^ne^{-\th (k+j\dl)}\}.
\end{align}
(b) The fact that, for any $n\ge 0$, any pair $\nx$ and $k\ge 0$, 
$V_k^{n}(\nu,\xi)\le V_k^{n+1}(\nu,\xi)$ can be obtained by induction since, obviously, for any pair $\nx$ and $k\in \bbn$, 
$V_k^{0}(\nu,\xi)\le V_k^{1}\nx$. Finally, it is enough to take into account the definition of $V_k^{n}(\nu,\xi)$, $n\in \bbn$.

\noindent (c) We now focus on the consistency condition. Once more it will be obtained by induction. So let us consider the case $n=0$. Let $\nu$ and $\n'$ two stopping times such that $\n\le \n'$ and $\xi$ be an $\U$-valued and $\F_\n$-measurable random variable.
For any $k\in \bbn$,
\begin{equation}\label{Y0appx1}
 V_{k}^{0}(\nu',\xi)= \mathbb{E}\left[\exp\left\lbrace \sum_{\ell\ge k}e^{-\th\ell }g(X_k+\xi)1_{\{\ell \ge \n'\}}\right\rbrace|\mathcal{F}_{k}\right]
\end{equation} and 
\begin{align*}
 V_{k}^{0}(\nu,\xi)&= \mathbb{E}\left[\exp\left\lbrace \sum_{\ell\ge k}e^{-\th\ell }g(X_k+\xi)1_{\{\ell \ge \n\}}\right\rbrace|\mathcal{F}_{k}\right]\\&
 =\mathbb{E}\left[\exp\left\lbrace \sum_{\ell\ge k}e^{-\th\ell }
 g(X_k+\xi)1_{\{\ell\ge\n' \}}\right\rbrace \exp\left\lbrace \sum_{\ell\ge k}e^{-\th\ell }
 g(X_k+\xi)1_{\{\n\le \ell<\n' \}}\right\rbrace|\mathcal{F}_{k}\right].
\end{align*}
Therefore, for any $k\ge 0$, 
$$
V_{k\vee \n'}^{0}(\nu,\xi)=V_{k\vee \n'}^{0}(\nu',\xi).
$$
Next, assume that the property holds for some $n$, i.e., for any two stopping times $\n\le \n'$ and $\xi$ a $\U$-valued and $\F_\n$-measurable random variable we have  
$$V^{n}_k(\nu,\xi)=V^{n}_k(\nu',\xi), \quad \,\,k\ge \nu'.$$
Then we have: For any $k\ge \n'$,
\begin{align*}
\Theta_{k}^{n+1}(\nu,\xi)&= \max\limits_{\beta \in U} \bigg\lbrace  \mathbb{E}\bigg\lbrack \exp\bigg\lbrace \sum_{k+1\le \ell <k+\dl}e^{-\th \ell}g(X_\ell+\xi)1_{\{\ell \ge \n'\}}-e^{-\th(k+\Delta)}\Psi(\beta) \bigg\rbrace V_{k+\Delta}^{n}(\nu',\xi+\beta)|\mathcal{F}_{k}\bigg\rbrack \bigg\rbrace\\
&=\Theta_{k}^{n+1}(\nu',\xi).
\end{align*}
since, on the one hand, $1_{\{\ell \ge \n\}}=1_{\{\ell \ge \n'\}}+1_{\{\n'>\ell \ge \n\}}$ and, on the other hand, by using the induction assumption it holds that $V_{k+\Delta}^{n}(\nu,\xi+\beta)=V_{k+\Delta}^{n}(\nu',\xi+\beta)$, as $\beta$ is deterministic and then $\xi+\b$ is $\F_\n$-measurable. Going back to the definitions of $V_{k}^{n+1}(\nu,\xi)$ and $V_{k}^{n+1}(\nu',\xi)$ in \eqref{Yapp} we then obtain, for any $k\ge\n'$, 
$$V_{k}^{n+1}(\nu,\xi)=V_{k}^{n+1}(\nu',\xi)$$
which is the desired result and the consistency property holds for every $n\ge 0$. 
\end{proof}
Next, we have the following properties.
\begin{proposition}\label{conv-rs} There exists a constant $C$ such that, for any pair 
$(\nu,\xi)$ and any $k\ge 0$,
\begin{equation}\label{V-n-k-bound}
|V^n_k(\nu,\xi)|\le C,\qquad k\ge 0.
\end{equation}
Consequently, the sequence of processes $(V_k^n(\nu,\xi))_{k\ge 0}$ converges $\P$-a.s. Moreover, the limit process 
$$
V_k(\nu,\xi):=\lim_{n\rw \infty}V^n_k(\nu,\xi),\quad k\ge 0,
$$
satisfies
\begin{equation}\label{V-k-nu}
V_k(\nu,\xi)=\esssup_{\t\ge k}
\E\bigg\lbrack\exp\bigg\lbrace \sum_{k\le \ell\le\t}e^{-\th \ell}g(X_\ell+\xi)1_{\{\ell \ge \n\}}\bigg\rbrace\co_\t(\nu,\xi)|\Fc_k\bigg\rbrack, \quad k\ge 0,
\ee
where
\begin{equation}\label{obst-V}
\begin{array}{lll}
\Theta_k(\nu,\xi)= \max\limits_{\beta \in U} \left\lbrace  \mathbb{E}\left\lbrack \exp\left\lbrace \underset{k+1\le \ell <k+\dl}{\sum}\,e^{-\th \ell}g(X_\ell+\xi)1_{\{\ell \ge \n\}} \right.\right. \right. \\ \left. \left. \left. \qquad\qquad\qquad \qquad\qquad\qquad \qquad\qquad  -e^{-\th(k+\Delta)}\Psi(\beta) \right\rbrace V_{k+\Delta}(\nu,\xi+\beta)|\mathcal{F}_{k}\right\rbrack \right\rbrace,\quad k\ge 0. 
\end{array}
\end{equation}
Furthermore, 
for any two $\F$-stopping times $\nu$ and $\nu^{\prime}$ such that $\nu\leq \nu^{\prime}$ and any $\mathcal{F}_{\nu}$-measurable r.v. $\xi$, we have $\mathbb{P}$-a.s.,
\be\lb{consist-rs}
V_k(\nu, \xi)=V_k(\nu^{\prime}, \xi),\qquad k\geq \nu'. 
\ee
Finally, \be\lb{V-infty}
V_{\infty}^-(\nu,\xi)=\lim_{k\rwi}V_k(\nu,\xi)\le 1=\Theta_{\infty}(\nu,\xi)=V_{\infty}(\nu,\xi).
\ee
\end{proposition}
\def \rw{\rightarrow}
\def \rwi{\rightarrow \infty}
\begin{proof} The sequence of processes $(V^n\nx)_{n\ge 0}$ is non-decreasing and we know, by \eqref{estvn}, that 
$\forall n\ge 0$, $\forall \nx$, $\forall k\in \bbn,$
\begin{align}\lb{estvn2}
|V_k^{n}(\nu,\xi)|&\le 
\exp\{\|g\|_{\infty}(1-e^{-\th})^{-1}e^{-\th k}+ \|\Psi\|_\infty\sum_{j=1}^ne^{-\th (k+j\dl)}\}\nn\\&
\le 
\exp\{\|g\|_{\infty}(1-e^{-\th})^{-1}e^{-\th k}+ \|\Psi\|_\infty e^{-\th (k+\dl)}(1-e^{-\theta\dl})^{-1}\}\nn\\
&
\le 
\exp\{\|g\|_{\infty}(1-e^{-\th})^{-1}+ \|\Psi\|_\infty e^{-\th \dl}(1-e^{-\theta\dl})^{-1}\}:=C
\end{align}
Therefore, the limiting process $V_k\nx:=\lim_{n\rw \infty}V^n_k\nx$, $k\in \bbn$, exists and for any $k\in \bbn$, 
\begin{align}
\lb{estv1}|V_k(\nu,\xi)|&
\le 
\exp\{\|g\|_{\infty}(1-e^{-\th})^{-1}e^{-\th k}+ \|\Psi\|_\infty e^{-\th (k+\dl)}(1-e^{-\theta\dl})^{-1}\}\\
&
\le 
\exp\{\|g\|_{\infty}(1-e^{-\th})^{-1}+ \|\Psi\|_\infty e^{-\th \dl}(1-e^{-\theta\dl})^{-1}\}:=C.\nn
\end{align}
Next, since $(V^n_k\nx)_{k\ge 0}$ satisfies \eqref{Yapp} and is non-decreasing in $n$, then $(V_k\nx)_{k\ge 0}$ satisfies \eqref{V-k-nu}, since $\U$ is finite. Furthermore, note that 
$
\lim_{k\rwi}V_k(\nu,\xi)$ exists since 
$(V_k(\nu,\xi).\exp\lbrace \sum_{1\le \ell<k}e^{-\th \ell}g(X_\ell+\xi)1_{\{\ell \ge \n\}}\rbrace)_{k\ge 0}$ is a bounded supermartingale and 
$(\exp\lbrace \sum_{1\le \ell<k}e^{-\th \ell}g(X_\ell+\xi)1_{\{\ell \ge \n\}}\rbrace)_{k\ge 0}$ is convergent w.r.t. $k$ and its limit is positive and bounded. Now by \eqref{estv1}, taking the limit as $k\rwi$, we obtain \eqref{V-infty}. Finally, the consistency condition \eqref{consist-rs} is an immediate consequence of \eqref{consiscdt} satisfied by $V\nx$.

\end{proof}
\medskip
We end this section with its main result of this section whose proof is similar to that of Theorem \ref{opt-strat-rn}. 

\begin{theorem}\lb{opt-strat-rs} The discrete time risk-sensitive impulse control with delay problem admits an optimal strategy $\d^*=(\t_n^*,\xi_n^*)_{n\ge 0}$. Moreover,
$$
V_0(0,0):=V_0(\n,\xi)_{|(\n,\xi)=(0,0)}=\sup_{\d \in \ca}J(\d)=J(\d^*).
$$
More precisely, such a strategy $\d^*=(\t_n^*,\xi_n^*)_{n\ge 0}$ is defined as follows.

\nd (i) $$
\t^*_0=\inf\{\ell \ge \Delta,\,\, 
V_\ell(0,0)=\co_\ell(0,0)\}.
$$
(ii) $\xi_0^*$ is the $\U$-valued and $\Fc_{\t_0^*}$-measurable r.v. such that 
\begin{equation}\lb{eqsimplifo-rs}\begin{array}{lll}
\max_{\beta \in \U}\E\left[\exp{\left\{-e^{-\th (\t^*_0+\dl)}\Psi(\beta)\right\}}V_{\t^*_0+\dl}(0,\beta)|\Fc_{\t^*_0}\right]
 \\ \qquad\qquad =\max_{\beta \in \U}\E\left[\exp{\left\{-e^{-\th (\t^*_0+\dl)}\Psi(\beta)\right\}}V_{\t^*_0+\dl}(\t^*_0,\beta)|\Fc_{\t^*_0}\right]\nn\\ \qquad\qquad\qquad\qquad\qquad =
\E\left[\exp{\left\{-e^{-\th (\t^*_0+\dl)}\Psi(\xi_0^*)\right\}}V_{\t^*_0+\dl}(\t^*_0,\xi_0^*)|\Fc_{\t^*_0}\right].
\end{array}
\end{equation}
 \nd(iii) For $n\ge 1$,
$$
\t^*_n=\inf\{\ell \ge \Delta+\t^*_{n-1},\,\, 
V_\ell(\t^*_{n-1},\xi_0^*+\ldots+\xi_{n-1}^*)=\co_\ell(\t^*_{n-1},\xi_0^*+\ldots+\xi_{n-1}^*)\}
$$
(iv) $\xi_n^*$ is the $\U$-valued and $\Fc_{\t_0^*}$-measurable r.v. such that
\begin{align}
&\max_{\beta \in \U}\E\left[\exp{\left\{-e^{-\th (\t^*_n+\dl)}\Psi(\beta)\right\}} V_{\t^*_n+\dl}(\t^*_n,\xi_0^*+\ldots+\xi_{n-1}^*+\beta)|\Fc_{\t^*_n}\right]=
\nn \\&\qquad\qquad \qquad \qquad \E\left[\exp{\left\{-e^{-\th (\t^*_n+\dl)}\Psi(\xi_n^*)\right\}}V_{\t^*_n+\dl}(\t^*_n,\xi_0^*+\ldots+\xi_{n-1}^*+\xi_n^*)|\Fc_{\t^*_n}\right].
\end{align}
\end{theorem}

 Note that the stopping times $\t_n^*$ are optimal because, by \eqref{V-infty},  $\lim_{k\to \infty}V_k\nx\le V_\infty \nx=1$ (see \cite{hamahassa}, p.184).

\section*{Conclusion}

In this paper, we addressed a class of infinite-horizon stochastic impulse control problems in discrete time, with execution delays. Using probabilistic tools, particularly the notion of the Snell envelope of processes, we constructed an optimal control strategy among all admissible strategies for both risk-neutral and risk-sensitive utility functions. Furthermore, we proved the existence of bounded $\varepsilon$-optimal strategies, which is very useful in settings where exact optimal strategies may not be easily implementable. Our approach bridges the gap between theoretical impulse control models and real-world decision-making processes, where delays are inevitable.
Extending the current framework to incorporate random execution delays, state-dependent delays, or systems with constraints on intervention frequency could provide further insight.

\end{document}